\documentclass[10pt]{amsart}

\usepackage{amsmath,xspace,amssymb,mathrsfs}
\usepackage{color}
\usepackage{etoolbox}

\newcommand{\N}{\mathbb{N}}
\newcommand{\m}{\mathfrak{m}}

\DeclareMathOperator{\Spec}{Spec}
\DeclareMathOperator{\Max}{Max}
\DeclareMathOperator{\Hom}{Hom}
\DeclareMathOperator{\gr}{gr}

\newtheorem{proposition}{Proposition}[section]
\newtheorem{lemma}[proposition]{Lemma}

\newtheorem{corollary}[proposition]{Corollary}
\newtheorem{theorem}[proposition]{Theorem}

\theoremstyle{definition}
\newtheorem{definition}[proposition]{Definition}
\newtheorem{example}[proposition]{Example}

\newtheorem{remark}[proposition]{Remark}
\newtheorem{conjecture}[proposition]{Conjecture}

\begin{document}

\title[Ideal avoidance property]{On the ideal avoidance property}

\author[J. Chen, A. Tarizadeh]{Justin Chen, Abolfazl Tarizadeh}
\address{Institute for Computational and Experimental Research in Mathematics, Providence, RI, 02903 USA}
\email{justin\_chen2@brown.edu}
\address{Department of Mathematics, Faculty of Basic Sciences, University of Maragheh,
P.O. Box 55136-553, Maragheh, Iran.}
\email{ebulfez1978@gmail.com \\
atarizadeh@maragheh.ac.ir}

\date{}
\subjclass[2010]{13A15, 13C05, 14A05}
\keywords{Generalized Prime Avoidance Lemma; ideal avoidance property (avoidance ideal); avoidance ring}

\begin{abstract} In this article, we investigate the avoidance property of ideals and rings.
Among the main results, a general version of the avoidance lemma is formulated.
It is shown that every idempotent ideal (and hence every pure ideal) has avoidance. The avoidance property of arbitrary direct products of avoidance rings is characterized.
It is shown that every overring of an avoidance domain is an avoidance domain. Next, we show that every avoidance $\mathbb{N}$-graded ring whose base subring is a finite field is a PIR. It is also proved that the avoidance property is preserved under flat ring epimorphisms.
Dually, we formulate a notion of strong avoidance, and show that it is reflected by pure morphisms.
\end{abstract}

\maketitle

\section{Introduction and Preliminaries}

In a classical article \cite{McCoy}, McCoy proved a surprising and vast generalization of the prime avoidance lemma. The core of his result asserts that if an ideal $I$ of a commutative ring $R$ is contained in the union of finitely many ideals of $R$, then some (positive) power of $I$ is contained in one of them.
Additionally, the thesis \cite{Quartararo} as well as the article \cite{Quartararo-Butts} investigates $u$-ideals and $u$-rings (two basic notions closely related to the topic of McCoy's result), and several deep and technical results were obtained. However, all of these important and general results have remained relatively unknown over the years. \\

The main goals of this article are the continuation of the above studies, as well as bringing them to the attention of the mathematical community. In this article, all rings are commutative with $1 \neq 0$.
We say that an ideal $I$ of a ring $R$ has \emph{the ideal avoidance property} or simply has \emph{avoidance} (or, \emph{u-ideal} in the sense of \cite{Quartararo-Butts}) if whenever $I_{1},\ldots,I_{n}$ are finitely many ideals of $R$ with $I \subseteq \bigcup\limits_{k=1}^{n}I_{k}$, then $I\subseteq I_{k}$ for some $k$.
Clearly an ideal $I$ has avoidance if and only if $I= \bigcup\limits_{k=1}^{n}I_{k}$ implies that $I=I_{k}$ for some $k$.
If every ideal of a ring $R$ has avoidance, then we say that $R$ is an \emph{avoidance ring}
(or, \emph{u-ring} in the sense of \cite{Quartararo-Butts}). It suffices to check this for finitely generated ideals: if every finitely generated ideal of a ring $R$ has avoidance, then $R$ is an avoidance ring (see \cite[Proposition 1.1]{Quartararo-Butts}). \\

The main contributions of this article are as follows: In Theorem \ref{Theorem 7 McCoy}(iv), a general version of the avoidance lemma is formulated.
It is shown that every idempotent ideal has avoidance (see Corollary \ref{Lemma 2 McCoy}).
In Corollary \ref{Corollary 5 fraction}, it is proved that every invertible fractional ideal has avoidance which improves \cite[Theorem 1.5]{Quartararo-Butts}.
The avoidance property of arbitrary direct products of avoidance rings is characterized (see Corollary \ref{Corollary infinite avoidance}).
In particular, we give two alternative proofs to the fact that a finite product of avoidance rings is always an avoidance ring.
In Theorem \ref{Theorem QBB}, it is shown that every overring of an avoidance domain is an avoidance domain.
We also investigate the behavior of the avoidance property along various constructions and ring maps in Theorems \ref{thm:graded}, \ref{Theorem 2}, Lemma \ref{Theorem 3} and their consequences.
In particular, we show that every avoidance $\mathbb{N}$-graded ring whose base subring is a finite field is a principal ideal ring (PIR). 
To prove this theorem, we first obtain a general result (see Lemma \ref{lemma T-C}) which asserts that every $\mathbb{N}$-graded ring whose base subring is a field $K$ has a presentation of the form $K[X]/I$ where $X$ is a set of indeterminates over $K$ and $I$ is an ideal of the polynomial ring $K[X]$ which is contained in the square of the maximal ideal $(X)$.  
Next, it is shown that the avoidance property is preserved under flat ring epimorphisms.
We also observe that the avoidance property need not be reflected along monomorphisms (see Example \ref{Example 1}).
This failure leads us to a notion of strong avoidance, and we show that strong avoidance is reflected along (cyclically) pure morphisms.

\section{The general avoidance lemma}

In this section we first reformulate McCoy's Theorem \cite{McCoy}, the most general version of the avoidance lemma that encompasses all known conditions under which avoidance holds, and then add (iv) as a new observation.

\begin{theorem}\label{Theorem 7 McCoy} Let $I$ and $I_{1},\ldots,I_{n}$ be finitely many ideals of a ring $R$ with $I\subseteq\bigcup\limits_{k=1}^{n}I_{k}$.
Then the following assertions hold.
\\
$\mathbf{(i)}$ If $n\leqslant2$ then $I\subseteq I_{1}$ or $I\subseteq I_{2}$.
\\
$\mathbf{(ii)}$ If $I$ can be covered by no proper subset of $\{I_{1},\ldots,I_{n}\}$ then there exists a positive integer $d\geqslant1$ such that $I^{d} \subseteq \bigcap\limits_{k=1}^{n}I_{k}$.
\\
$\mathbf{(iii)}$ There exists a positive integer $d\geqslant1$ such that $I^{d}\subseteq I_{k}$ for some $k$.
\\
$\mathbf{(iv)}$ If all but two of the $I_k$ are radical ideals, then $I \subseteq I_{k}$ for some $k$.
\end{theorem}

\begin{proof} (i): It is clear.
\\
(ii): See the proof of \cite[Theorem 1]{McCoy}. \\
(iii): We may find a nonempty set $S$ in $\{I_{1},\ldots,I_{n}\}$ which covers $I$, but
no proper subset of $S$ covers $I$.
Applying (ii) for the set $S$, there exists a positive integer $d\geqslant1$ such that $I^{d}\subseteq I_{k}$ for all $I_{k}\in S$.
\\
(iv): Suppose $I$ is not contained in any of the ideals $I_{k}$. There exists a set $S$ in $\{I_{1},\ldots,I_{n}\}$ which covers $I$, but no proper subset of $S$ covers $I$. Then by (ii), there exists a positive integer $d\geqslant1$ such that $I^{d}\subseteq I_{s}$ where $I_{s}\in S$ is a radical ideal, because $|S|\geqslant3$.
It follows that $I \subseteq \sqrt{I}=\sqrt{I^{d}} \subseteq \sqrt{I_{s}}=I_{s}$ which is a contradiction.
\end{proof}

In Theorem \ref{Theorem 7 McCoy}(ii), it can be shown that the optimal bound for the power $d$ is $n-1$. For its proof see \cite{Backelin}. \\

Clearly every principal ideal has avoidance.
The following result and Corollary \ref{Corollary 5 fraction} provide interesting examples of ideals which have avoidance.

\begin{corollary}\label{Lemma 2 McCoy} Every idempotent ideal has avoidance.
\end{corollary}

\begin{proof} Let $I$ be an idempotent ideal ($I=I^{2}$) of a ring $R$.
Let $I_{1},\ldots,I_{n}$ be finitely many ideals of $R$ with $I \subseteq \bigcup\limits_{k=1}^{n}I_{k}$.
Then by Theorem \ref{Theorem 7 McCoy}(iii), there exists some $d\geqslant1$ such that $I=I^{d}\subseteq I_{k}$ for some $k$.
\end{proof}

Every ideal generated by a set of idempotents is an idempotent ideal and hence has avoidance.
Although every finitely generated idempotent ideal is generated by an idempotent element, an arbitrary idempotent ideal is not necessarily generated by idempotent elements.
Also note that if $I$ is a pure ideal of a ring $R$ (i.e.
for each $a\in I$ there exists some $b\in I$ such that $a=ab$, or equivalently, $R/I$ is a flat $R$-module) then it is an idempotent ideal and hence has avoidance.

\begin{remark} Note that Theorem \ref{Theorem 7 McCoy}(iii) in particular yields that $I \subseteq \sqrt{I_k}$ for some $k$, which gives an alternative proof of our recent result \cite[Theorem 2.2]{Tarizadeh-Chen}.
\end{remark}

\section{The ideal avoidance property}

We begin this section with a technical result which is useful in studying the avoidance property.

\begin{lemma}\label{Lemma 4 Q-B} Let $I_{1},\ldots,I_{n}$ be finitely many ideals of a ring $R$ and $M$ a finitely generated faithful $R$-module.
If $M=\bigcup\limits_{k=1}^{n}I_{k}M$, then $I_{k}=R$ for some $k$.
\end{lemma}

\begin{proof} It is proved exactly like \cite[Proposition 1.4]{Quartararo-Butts} by induction on $n$.
\end{proof}

One can naturally generalize the notion of the ideal avoidance property to modules (see also \cite{Quartararo-Butts}): for a ring $R$, we say that an $R$-module $M$ has \emph{avoidance} if whenever $M$ is equal to a finite union of $R$-submodules, then $M$ must be equal to one of them.
If $N \subseteq M$ is an $R$-submodule, then $N$ has avoidance if and only if whenever $N$ is contained in a finite union of $R$-submodules of $M$, then $N$ is contained in one of them. \\

In dealing with union of ideals, it is important to notice that if $\{ I_{k} \}$ is a (finite or infinite) family of ideals of a ring $R$ and $I$ is an ideal of $R$ which is generated by a subset of $\bigcup\limits_{k}I_{k}$, then in general, we cannot deduce that $I$ is contained in $\bigcup\limits_{k}I_{k}$.
The same holds for submodules of a module. \\

Recall that for a given ring $R$, an $R$-submodule $I$ of $T(R)$ (the total ring of fractions of $R$) is said to be a \emph{fractional ideal} of $R$ if $aI\subseteq R$ for some nonzerodivisor $a\in R$.
If $I$ and $J$ are fractional ideals of $R$ then $IJ$, the set of all finite sums $\sum\limits_{i=1}^{n}x_{i}y_{i}$ with $x_{i}\in I$ and $y_{i}\in J$, is also a fractional ideal of $R$.
A fractional ideal $I$ of $R$ is said to be invertible if $IJ=R$ for some fractional ideal $J$ of $R$.
The following result slightly improves \cite[Theorem 1.5]{Quartararo-Butts}.

\begin{corollary}\label{Corollary 5 fraction} Every invertible fractional ideal has avoidance.
\end{corollary}

\begin{proof} Let $I$ be an invertible fractional ideal of a ring $R$ and
let $M_{1},\ldots,M_{n}$ be finitely many $R$-submodules of $T(R)$ with $I= \bigcup\limits_{k=1}^{n}M_{k}$.
Since $I$ is invertible, there exists an $R$-submodule $J$ of $T(R)$ such that $IJ=R$.
Then for each $k$, we have $M_{k}=RM_{k}=I(JM_{k})$ and $JM_{k}$ is an ideal of $R$, because $JM_{k}\subseteq JI=R$.
So $I= \bigcup\limits_{k=1}^{n}I(JM_{k})$.
Now every invertible ideal is finitely generated, and also faithful (since it contains a nonzerodivisor of $R$).
Hence $JM_{k}=R$ by Lemma \ref{Lemma 4 Q-B}, and so $I=IR=IJM_{k}=RM_{k}=M_{k}$ for some $k$.
\end{proof}

By an \emph{invertible module} we mean a module $M$ over a ring $R$ such that there exists an $R$-module $N$ with $M \otimes_{R} N \cong R$ as $R$-modules.
It is well known that an $R$-module $M$ is invertible if and only if $M$ is a finitely generated projective $R$-module of constant rank 1, or equivalently, the canonical map $M\otimes_{R}M^{\ast}\rightarrow R$ given by $m\otimes f\mapsto f(m)$
is an isomorphism, where $M^{\ast}=\Hom_{R}(M,R)$ is the dual of $M$.
In this regard, we have the following result.

\begin{corollary} Let $R$ be a ring such that the Picard group of its total ring of fractions is trivial.
Then every invertible $R$-module has avoidance.
\end{corollary}

\begin{proof} From hypothesis and \cite[Theorem 4.9]{Tarizadeh yalniz} we obtain that every invertible $R$-module is isomorphic to an invertible fractional ideal of $R$.
Then apply Corollary \ref{Corollary 5 fraction}.
\end{proof}

In relation with the above result, for a given ring $R$, if $T(R)$ has finitely many maximal ideals (e.g. $R$ is reduced with finitely many minimal primes), then the Picard group of $T(R)$ is trivial. For the details see \cite[Corollaries 4.10, 4.11]{Tarizadeh yalniz}.\\

Recall that a \emph{B\'{e}zout ring} is a ring such that every finitely generated ideal is principal.
It is not hard to see that the class of Bezout rings is stable under taking quotients, localizations and finite products.
Every PIR is a B\'{e}zout ring.
Also recall that a \emph{Pr\"{u}fer domain} is a ring $R$ such that every nonzero finitely generated ideal of $R$ is invertible. (Note that this definition implies that $R$ is a domain: if $ab=0$ for some $a,b\in R$ with $a \ne 0$, then $Ra$ is invertible, so there exists an $R$-submodule $I \subseteq T(R)$ such that $(Ra)I=R$. It follows that $I=Rx$ for some $x \in T(R)$. Then $1=(ra)(r'x)$ for some $r, r' \in R$ and so $b=(rab)(r'x)=0$.) \\

In the following result we list some of the major examples of avoidance rings, all of which are well known and can be found in \cite{Quartararo-Butts} or \cite{Quartararo}.

\begin{corollary} \label{cor:exsAvoidance}
If for a ring $R$ any of the following conditions holds, then $R$ is an avoidance ring.
\\
$\mathbf{(i)}$ $R$ is a B\'{e}zout ring.
\\
$\mathbf{(ii)}$ $R$ is a Pr\"{u}fer domain.
\\
$\mathbf{(iii)}$ $R$ contains an infinite field as a subring.
\end{corollary}

\begin{proof}
(i): It is easy, see also \cite[Proposition 1.1(2)]{Quartararo-Butts}.
\\
(ii): It follows from Corollary \ref{Corollary 5 fraction}.
\\
(iii): It is an immediate consequence of \cite[Proposition 1.7]{Quartararo-Butts}.
\end{proof}

Every Boolean ring, and more generally every absolutely flat (i.e. von Neumann regular) ring, as well as every valuation ring, is a B\'{e}zout ring, hence is an avoidance ring.
In particular, any reduced Artinian ring is absolutely flat, so is an avoidance ring.
Also, as Pr\"{u}fer domains are a natural generalization of Dedekind domains to the non-Noetherian setting, every Dedekind domain is an avoidance ring.
\\

Our next goal is to study when the avoidance property is preserved by various ring-theoretic operations.
Let $\phi : R \to S$ be a ring map.
Recall that an ideal $J$ of $S$ is \emph{extended} under $\phi$ if $J = I^e$ for some ideal $I$ of $R$, or equivalently, $J^{ce} = J$.
Dually, an ideal $I$ of $R$ is \emph{contracted} under $\phi$ if $I = J^c$ for some ideal $J$ of $S$, or equivalently, $I = I^{ec}$.

\begin{lemma} \label{Lemma 1} Let $\phi : R \to S$ be a ring map such that every ideal of $S$ is extended under $\phi$.
If $R$ is an avoidance ring, then $S$ is as well.
\end{lemma}

\begin{proof} If $J, J_1,\ldots,J_n$ are finitely many ideals of $S$ with $J \subseteq \bigcup\limits_{k=1}^{n}J_k$, then $\phi^{-1}(J) \subseteq \phi^{-1}(\bigcup\limits_{k=1}^{n}J_k)=
\bigcup\limits_{k=1}^{n}\phi^{-1}(J_k)$.
Since $R$ has avoidance, $J^{c}=\phi^{-1}(J) \subseteq (J_k)^{c}=\phi^{-1}(J_k)$ for some $k$.
Thus $J=J^{ce} \subseteq (J_k)^{ce}=J_k$.
\end{proof}

\begin{corollary}\cite[Proposition 1.3]{Quartararo-Butts} \label{Corollary 1}
If $R$ is an avoidance ring, then any quotient or localization of $R$ is an avoidance ring.
\end{corollary}

\begin{proof}
It follows from Lemma \ref{Lemma 1}.
\end{proof}

We rediscovered the following result independently of \cite[Theorem 1.18]{Quartararo}:

\begin{theorem}\label{Theorem 4} If $R = \prod_{i=1}^n R_i$ is a finite product of rings, then $R$ is an avoidance ring if and only if $R_i$ is an avoidance ring for all $i$.
\end{theorem}

\begin{proof}
By Corollary \ref{Corollary 1}, if $R$ is an avoidance ring then each $R_i$ is as well.
Conversely, by induction it suffices to consider a product of two rings.
Suppose the rings $R_{1}, R_{2}$ are avoidance rings.
It is well known that every ideal of $R_{1}\times R_{2}$ is precisely of the form $I \times J$ where $I$ (resp.
$J$) is an ideal of $R_{1}$ (resp.
$R_{2}$).
Let $I \times J$ be an ideal in $R_{1}\times R_{2}$ with $I \times J \subseteq \bigcup\limits_{i=1}^n (I_i \times J_i)$.
It follows that $I \subseteq \bigcup\limits_{i=1}^{n}I_i$.
Since $R_{1}$ is an avoidance ring, there exists at least one $i$ such that $I \subseteq I_i$.
By reordering we may assume there exists some $d$ with $1\leqslant d\leqslant n$ such that that $I \subseteq I_i$ for all $1\leqslant i\leqslant d$ and $I \not \subseteq I_i$ for all $d+1\leqslant i \leqslant n$.
Then we claim that $I\times J \subseteq \bigcup\limits_{i=1}^{d}I_{i}\times J_{i}$.
Indeed, suppose $I \times J \not \subseteq \bigcup\limits_{i=1}^d (I_i \times J_i)$, thus $d<n$.
Choose $(a,b) \in (I \times J) \setminus \bigcup\limits_{i=1}^d (I_i \times J_i)$.
Since each of $I_1, \ldots, I_d$ contains $I$ (and thus $a$), it follows that $b$ is not in any of $J_1, \ldots, J_d$.
Next, since $I$ is not contained in any of $I_{d+1}, \ldots, I_n$, there exists $c \in I \setminus \bigcup\limits_{i=d+1}^n I_i$ (since $R_{1}$ is an avoidance ring).
Then $(c,b) \in (I \times J) \setminus \bigcup_{i=1}^n (I_i \times J_i)$, contradiction.
This establishes the claim.
It follows that $J \subseteq \bigcup\limits_{i=1}^{d}J_{i}$.
Since $R_{2}$ is an avoidance ring, there exists $1\leqslant k\leqslant d$ such that $J \subseteq J_k$.
So $I \times J \subseteq I_k \times J_k$.
\end{proof}

The avoidance property is also preserved by certain infinite products.
In fact, every direct product of fields is an avoidance ring.
More generally, every direct product of absolutely flat rings is absolutely flat, and hence is an avoidance ring.
In this regard, see also Corollary \ref{Corollary infinite avoidance}.
\\

There are two minor mistakes in the formulation of \cite[Theorem 1.8]{Quartararo-Butts}: (i) the dimension of $V$ should be $\geqslant2$, and (ii) $M$ should be finitely generated.
We state the corrected form as follows:

\begin{lemma}\label{Lemma 3 Quartararo-Butts} Let $\mathfrak{m}$ be a maximal ideal of a ring $R$ such that $F=R/\mathfrak{m}$ is a finite field with $n-1$ elements ($n\geqslant3$) and $M$ a finitely generated $R$-module such that the $F$-vector space $V=M/\mathfrak{m}M$ has dimension $\geqslant2$.
Then there exist $n$ proper $R$-submodules $M_{1},\ldots, M_{n}$ of $M$ such that $M=\bigcup\limits_{i=1}^{n}M_{i}$.
\end{lemma}

\begin{proof} See the proof of \cite[Theorem 1.8]{Quartararo-Butts}, taking into account that before the definition of $E_{2+i}$ we need to assume that $F\setminus\{0\}=\{x_{1},\ldots,x_{n-2}\}$, and moreover $q-1$ should be replaced with $n-2$.
\end{proof}

\begin{corollary}\cite[Corollary 1.9(ii)]{Quartararo-Butts}\label{Corollary 4 Q-B} A local ring $R$ with maximal ideal $\mathfrak{m}$ is an avoidance ring if and only if $R/\mathfrak{m}$ is infinite or $R$ is a B\'{e}zout ring.
\end{corollary}

\begin{proof} The implication ``$\Rightarrow$" follows from Lemma \ref{Lemma 3 Quartararo-Butts} and the fact that over a local ring $R$, a finitely generated $R$-module $M$ has a minimal generating set with $d$ elements if and only if the $R/\mathfrak{m}$-vector space $M/\mathfrak{m}M$ has dimension $d$.
The reverse implication follows from \cite[Proposition 1.7]{Quartararo} and Corollary \ref{cor:exsAvoidance}(i) (every B\'{e}zout ring has avoidance).
\end{proof}

There is also a minor mistake in formulation of \cite[Theorem 2.5]{Quartararo-Butts}: assuming $n > 1$ is wrong, i.e. it is possible that for every maximal ideal of $R$, the localization $R_M$ is not a B\'{e}zout ring.
In other words, it may happen that for each maximal ideal $M$ of $R$, the field $R/M$ is infinite.
The corrected formulation reads as follows.

\begin{lemma} Assume a ring $R$ has finitely many maximal ideals such that for each maximal ideal $M$, the field $R/M$ is infinite or $R_M$ is a B\'{e}zout ring.
Then $R$ is an avoidance ring.
\end{lemma}

\begin{proof} See the proof of \cite[Theorem 2.5]{Quartararo-Butts}.
\end{proof}

Note that Corollary \ref{Corollary 4 Q-B} holds more generally.
In fact, the following result is the culmination of avoidance rings.

\begin{theorem}\cite[Theorem 2.6]{Quartararo-Butts}\label{Theorem 6 Q-B} For a ring $R$ the following assertions are equivalent.
\\
$\mathbf{(i)}$ $R$ is an avoidance ring.
\\
$\mathbf{(ii)}$ Every finitely generated ideal of $R$ has avoidance.
\\
$\mathbf{(iii)}$ $R_{\mathfrak{p}}$ is an avoidance ring for all $\mathfrak{p}\in\Spec(R)$.
\\
$\mathbf{(iv)}$ $R_{M}$ is an avoidance ring for all $M\in\Max(R)$.
\\
$\mathbf{(v)}$ For each maximal ideal $M$ of $R$, the field $R/M$ is infinite or $R_M$ is a B\'{e}zout ring.
\end{theorem}

\begin{proof} The equivalence (i)$\Leftrightarrow$(ii) and the implications (i)$\Rightarrow$(iii)$\Rightarrow$(iv) are easy.
But the implications (iv)$\Rightarrow$(v)$\Rightarrow$(i) are nontrivial, for the details see \cite[Theorem 2.6]{Quartararo-Butts}, taking into account that if $M$ is a maximal ideal of a ring $R$ and $S \subseteq R$ a multiplicative subset with $M\cap S=\emptyset$, then $S^{-1}R/S^{-1}M\cong S^{-1}(R/M)$ is a field extension of $R/M$ and we have a canonical ring isomorphism $(S^{-1}R)_{S^{-1}M}\cong R_M$.
\end{proof}

By the above theorem, the avoidance is a local property. In particular, every Krull domain is an avoidance ring. 

\begin{corollary}\label{Corollary infinite avoidance} Let $\{ R_{i} \}$ be an arbitrary family of avoidance rings. Then $R = \prod\limits_{i}R_{i}$ is an avoidance ring if and only if $R$ modulo the ideal
$I = \bigoplus\limits_{i}R_{i}$ is an avoidance ring.
\end{corollary}

\begin{proof} The implication ``$\Rightarrow$" follows from Corollary \ref{Corollary 1}. Conversely, it suffices to show that $R$ satisfies Theorem \ref{Theorem 6 Q-B}(v).
Let $M$ be a maximal ideal of $R$.
First, suppose $I\subseteq M$.
Then by Theorem \ref{Theorem 6 Q-B}, the field $R/M$ is infinite or $(R/I)_{M/I}\cong(R/I)_{M}\cong R_{M}/IR_{M}$ is a B\'{e}zout ring.
Now $I$ is generated by the canonical idempotents $e_{k}=(\delta_{i,k})$ where $\delta_{i,k}$ is the Kronecker delta, so its extension $IR_{M}$ is generated by the elements $e_{k}/1$.
Since each $1-e_{k}\in R\setminus M$, one has $e_{k}/1=0$, thus $IR_{M}=0$, hence $R_M \cong (R/I)_{M/I}$.

Next, suppose $I \not \subseteq M$.
Then there exists some index $k$ such that the idempotent $e_{k} \in I \setminus M$.
Thus $1-e_{k}\in M$.
It follows that $M=\prod\limits_{i}M_{i}$ where $M_{k}=\pi_{k}(M)$ is a maximal ideal of $R_{k}$ and $M_{i}=R_{i}$ for all $i\neq k$ (here $\pi_{k}:R\twoheadrightarrow R_{k}$ is the canonical projection).
Now if $R/M\cong R_{k}/M_{k}$ is a finite field, then using Theorem \ref{Theorem 6 Q-B}, we have $R_{M}\cong(R_{k})_{M_{k}}$ is a B\'{e}zout ring.
\end{proof}


By an \emph{overring} of a ring $R$ we mean a subring of $T(R)$ which contains $R$.
All the claims in \cite[\S3]{Quartararo-Butts} have been proved in Quartararo's thesis \cite{Quartararo}, except the following one:

\begin{theorem}\label{Theorem QBB} Let $S$ be an overring of an integral domain $R$.
If $R$ is an avoidance ring, then $S$ is an avoidance ring.
\end{theorem}

\begin{proof} It suffices to show that $S$ satisfies Theorem \ref{Theorem 6 Q-B}(v).
Let $M$ be a maximal ideal of $S$.
If $S/M$ is a finite field, consider the ring extension $R/\mathfrak{m}\subseteq S/M$ with $\mathfrak{m}:=M\cap R$.
Then $R/\mathfrak{m}$ is a field, because every finite integral domain is a field, hence $\mathfrak{m}$ is a maximal ideal of $R$, so by hypothesis and Theorem \ref{Theorem 6 Q-B}, $R_{\mathfrak{m}}$ is a B\'{e}zout domain.
To conclude the assertion it suffices to show that $S_{M}$ is a B\'{e}zout domain.
Since $R$ is an integral domain, we have the canonical ring extensions $R_{\mathfrak{m}}\subseteq S_{M}\subseteq K$ where $K = T(R) = T(R_{\mathfrak{m}})$.
Thus $S_{M}$ is an overring of $R_{\mathfrak{m}}$, and it is well known that every overring of a B\'{e}zout domain is a B\'{e}zout domain (see \cite{Beauregard}).
\end{proof}

Recall that an ideal $I$ of a ring $R$ is called a \emph{multiplication ideal} if every ideal of $R$ contained in $I$ is of the form $IJ$ for some ideal $J$ of $R$.
If every ideal of a ring $R$ is a multiplication ideal then $R$ is called a \emph{multiplication ring}.
For example, every PIR is a multiplication ring.
Every pure ideal of a ring is also a multiplication ideal, since if $J$ is an ideal contained in a pure ideal $I$, then $J=IJ$.
In particular, every absolutely flat ring is a multiplication ring.
Multiplication ideals (and rings) have been studied in the literature over the past decades. In this regard, we obtain the following results.

\begin{corollary}\label{Corollary 6 multp} If every finitely generated ideal of an integral domain $R$ is a multiplication ideal, then $R$ is an avoidance ring.
\end{corollary}

\begin{proof} It suffices to show that every nonzero finitely generated ideal $I$ of $R$ has avoidance.
If $I=\bigcup\limits_{k=1}^{n}I_{k}$ then each $I_{k}=IJ_{k}$ for some ideal $J_{k}$ of $R$.
But $I$ is faithful, since $R$ is an integral domain.
Thus by Lemma \ref{Lemma 4 Q-B}, $J_{k}=R$ and so $I=I_{k}$ for some $k$.
\end{proof}

\begin{corollary} Every multiplication domain is an avoidance ring.
\end{corollary}

\begin{proof} It is an immediate consequence of Corollary \ref{Corollary 6 multp}.
\end{proof}

Motivated by the definition of a multiplication ideal, by a \emph{multiplication module} we mean a module $M$ over a ring $R$ such that every $R$-submodule of $M$ is of the form $IM$ for some ideal $I$ of $R$. For example, every invertible fractional ideal is a multiplication module.

\begin{corollary}\cite[Proposition 11]{Gottlieb} Every finitely generated faithful multiplication module has avoidance.
\end{corollary}

\begin{proof} It is an immediate consequence of Lemma \ref{Lemma 4 Q-B}.
\end{proof}

We observed that every PIR is an avoidance ring.
In the presence of some finiteness assumptions, the converse can hold as well:

\begin{proposition}\cite[\S3]{Quartararo-Butts}
Let $R$ be a finite ring.
Then $R$ is an avoidance ring if and only if $R$ is a PIR.
\end{proposition}

\begin{proof}
Suppose $R$ is not a PIR, and let $I \subseteq R$ be a nonprincipal ideal.
Then $I = \bigcup\limits_{x \in I} Rx$ is a finite union of proper principal ideals, so $R$ is not an avoidance ring.
\end{proof}

\begin{example} \label{ex:nonavoidance} It is useful to have a concrete example of a ring that is not an avoidance ring.
Let $K$ be a finite field.
Then the ring $K[x,y]/(x^2,xy,y^2)$ is a typical example of a finite ring with a nonprincipal ideal $I=(x,y)$, hence it is not an avoidance ring.
More generally, let $X=\{x_{i}: i\in S\}$ be any set of indeterminates over $K$ with $|S|\geqslant2$.
Consider $R=K[X]/\mathfrak{m}^{2}$ where $\mathfrak{m}=(x_{i}: i\in S)$  is the homogeneous maximal ideal of $K[X]$, and the nonprincipal ideal $(x_{1},x_{2}) \subseteq R$, which can be written as a finite union of principal ideals
$(x_1, x_2) = \bigcup\limits_{a,b \in K} (ax_{1}+bx_{2})$.
Hence, $R$ is not an avoidance ring. This example also shows that there are prime ideals which do not have avoidance.
\end{example}

\begin{example} \label{ex:nonavoidance2}  Here we give an example of a ring which is not an avoidance ring, although every prime ideal has avoidance.
For $K$ a field, consider the ring
\[
K[x^{1/2^\infty}] := K[x, x^{1/2}, x^{1/4}, \ldots] \cong K[x_1, x_2, \ldots]/(x_i - x_{i+1}^2 \mid i \geqslant 1).
\]
Let $S := K[x^{1/2^\infty}]/(x^2)$, and $R_n := S^{\otimes n} = S \otimes_K \ldots \otimes_K S$ for any $n \in \N$.
Then $R_n$ has only one prime ideal $\m$, which moreover satisfies $\m^2 = \m \ne 0$ and $R_n/\m \cong K$ (so that $\m$ has avoidance by Corollary \ref{Lemma 2 McCoy}): this can be seen from the presentation
\[
R_n \cong K[x_1, x_2, \ldots]/(x_1^2, \ldots, x_n^2, x_i - x_{i+n}^2 \mid i \geqslant 1).
\]
Note that $R_n$ is B\'ezout if and only if $n = 1$ (for $n > 1$, $(x_1, x_2)$ is a nonprincipal ideal).
In particular, if $K$ is finite, then by Theorem \ref{Theorem 6 Q-B}, $R_2$ is not an avoidance ring.
\end{example}

\begin{remark}\label{Remark 1 new} Let $K$ be a finite field.
Then the polynomial ring $K[x_{1},\ldots,x_{n}]$ is not an avoidance ring for all $n\geqslant2$, because its quotient modulo $\mathfrak{m}^{2}$ fails to have this property where $\mathfrak{m}=(x_{1},\ldots,x_{n})$.
Thus the avoidance property is not preserved by adjoining indeterminates.
For example, $K[x]$ is an avoidance ring, since it is a PID, but $K[x,y]$ is not an avoidance ring.
Furthermore, the avoidance property of rings need not pass to subrings or extensions.
As an example, for the ring $K[x,y]$, its subring $K$ and its field of fractions $K(x,y)$ both are avoidance rings.
The avoidance property need not be preserved by tensor products: take e.g.
$K[x,y]\cong K[x] \otimes_{K} K[y]$.
Finally, even if the quotient $R/\mathfrak{p}$ is an avoidance domain for each prime ideal $\mathfrak{p}$, then $R$ need not be an avoidance ring.
For instance, take the zero dimensional ring in Example \ref{ex:nonavoidance}.
\end{remark}

By a \emph{module-avoidance ring} (called
\emph{um-ring} in \cite{Quartararo-Butts}) we mean a ring $R$ such that every $R$-module has avoidance.
In \cite[Theorem 2.3]{Quartararo-Butts} one can find interesting characterizations for module-avoidance rings.
We slightly improve this result by adding the following equivalences.

\begin{proposition}\label{Proposition 2 $um$-ring} For a ring $R$ the following assertions are equivalent.
\\
$\mathbf{(i)}$ $R$ is a module-avoidance ring.
\\
$\mathbf{(ii)}$ Every $R$-algebra is an avoidance ring.
\\
$\mathbf{(iii)}$ Every ring extension of $R$ is an avoidance ring.
\\
$\mathbf{(iv)}$ Every finitely generated $R$-module has avoidance.
\end{proposition}

\begin{proof} The implications (i)$\Rightarrow$(ii)$\Rightarrow$(iii) and (i)$\Leftrightarrow$(iv) are clear.
\\
(iii)$\Rightarrow$(i): Let $M$ be an $R$-module and $M_{1},\ldots, M_{n}$ finitely many $R$-submodules of $M$ with $M=\bigcup\limits_{i=1}^{n}M_{i}$.
The Nagata idealization ring $R\times M$ is an extension of $R$ (notice that the map
$R\rightarrow R\times M$ given by $r \mapsto (r,0)$ is an injective ring map), thus by hypothesis it has avoidance.
Clearly $0\times M$ and $0\times M_{1},\ldots,0\times M_{n}$ are ideals of $R\times M$ with $0\times M=\bigcup\limits_{i=1}^{n}0\times M_{i}$.
So $0\times M=0\times M_{k}$ for some $k$, hence $M=M_{k}$.
\end{proof}

Every module-avoidance ring is an avoidance ring.
But the converse does not hold.
For example, the ring of integers $\mathbb{Z}$ is an avoidance ring, since it is a PID.
But it is not a module-avoidance ring, because the $\mathbb{Z}$-algebra $\mathbb{F}_{p}[x,y]$ is not an avoidance ring by Remark \ref{Remark 1 new} where $p$ is a prime number and $\mathbb{F}_{p}=\mathbb{Z}/p\mathbb{Z}$.
If $R$ is a module-avoidance ring and $R\rightarrow S$ is any ring map, then $S$ is a module-avoidance ring.
Every ring containing an infinite field as a subring is a module-avoidance ring.
By \cite[Theorem 2.3]{Quartararo-Butts}, a local ring is a module-avoidance ring if and only if its residue field is infinite (see also \cite[Proposition 1.7]{Quartararo-Butts}). 

\begin{corollary} Let $\{R_{i}\}$ be an arbitrary family of module-avoidance rings. Then $R=\prod\limits_{i}R_{i}$ is a module-avoidance ring if and only if $R$ modulo the ideal $I=\bigoplus\limits_{i}R_{i}$ is a module-avoidance ring.
\end{corollary}

\begin{proof} The implication ``$\Rightarrow$'' is clear. The reverse implication is proved exactly like Corollary \ref{Corollary infinite avoidance} by applying \cite[Theorem 2.3]{Quartararo-Butts} instead of Theorem \ref{Theorem 6 Q-B}.
\end{proof}

\begin{remark}
Let $R$ be a ring such that every ideal can be written as an intersection (possibly infinite) of primary ideals (e.g. a Noetherian or Laskerian ring).
Then $R$ is an avoidance ring if and only if it has ``primary avoidance'', i.e. if an ideal of $R$ is contained in a finite union of primary ideals, then it is contained in one of them.
Indeed, assume $I, I_1, \ldots,I_n$ are finitely many ideals of $R$ with $I \subseteq \bigcup \limits_{k=1}^n I_k$.
Suppose $I$ is not contained in any of the $I_{k}$.
Thus for each $k$ there exists a primary ideal $\mathfrak{q}_{k}$ of $R$ containing $I_{k}$ such that $I$ is not contained in $\mathfrak{q}_{k}$.
But $I \subseteq \bigcup\limits_{k=1}^{n}\mathfrak{q}_k$ which is a contradiction.
\end{remark}

Another setting in which the avoidance property is closely related to being a PIR is for graded algebras over finite fields (see Theorem \ref{thm:graded}).
For convenience, we include the following lemma:

\begin{lemma}\label{lemma T-C} Let $R=\bigoplus\limits_{n\geqslant0}R_{n}$ be an $\mathbb{N}$-graded ring with $R_{0}=K$ a field, $\mathfrak{m}=\bigoplus\limits_{n\geqslant1}R_{n}$ the irrelevant ideal and $S=\{a_{i}: i\in I\}$ a set of homogeneous elements of $R$ of positive degrees. Then the following statements are equivalent: \\
$\mathbf{(i)}$ $S$ is a minimal set of generators of the $R_{0}$-algebra $R$. \\
$\mathbf{(ii)}$ $S$ is a minimal set of generators of the ideal $\mathfrak{m}$. \\
$\mathbf{(iii)}$ The set $\{a_{i}+\mathfrak{m}^{2}: i \in I\}$ is a basis for the $K$-space $\mathfrak{m}/\mathfrak{m}^{2}$. \\
$\mathbf{(iv)}$ The canonical map of $K$-algebras $K[x_{i}: i \in I]\rightarrow R$ given by $x_{i}\mapsto a_{i}$ is surjective and its kernel is contained in the square of the ideal $(x_{i}: i\in I)$. \\
In particular, such a set $S$ always exists. 
\end{lemma}

\begin{proof} (i)$\Leftrightarrow$(ii): The assertion, including minimality, follows easily from \cite[Tag 07Z4]{de Jong}. \\
(ii)$\Leftrightarrow$(iii): It will be enough to show that $S$ generates the ideal $\mathfrak{m}$ if and only if $S':=\{a_{i}+\mathfrak{m}^{2}: i \in I\}$ generates the $R_{0}$-module $\mathfrak{m}/\mathfrak{m}^{2}$. First assume $S$ generates the ideal $\mathfrak{m}$. Then by \cite[Tag 07Z4]{de Jong}, we have $R=R_{0}[S]$. Thus each $b\in\mathfrak{m}$ can be written as a finite sum of elements of the form $ra^{m_{1}}_{i_{1}}\ldots a^{m_{n}}_{i_{n}}$ with $r\in R_{0}$, $m_{i}\geqslant1$ and $a_{i_{1}},\ldots,a_{i_{n}}\in S$. In this monomial, if some $m_{i}\geqslant2$ or we have $i_{d}\neq i_{k}$ for some indexes $i_{d}$ and $i_{k}$, then $ra^{m_{1}}_{i_{1}}\ldots a^{m_{n}}_{i_{n}}\in\mathfrak{m}^{2}$. Thus we may write $b+\mathfrak{m}^{2}=\sum\limits_{k=1}^{p}r_{k}(a_{k}+\mathfrak{m}^{2})$ where $r_{k}\in R_{0}$ and $a_{k}\in S$ for all $k$. Hence, $S'$ generates the $R_{0}$-module $\mathfrak{m}/\mathfrak{m}^{2}$. Conversely, we may write $\mathfrak{m}=\sum\limits_{i\in I}R_{0}a_{i}+\mathfrak{m}^{2}\subseteq \sum\limits_{i\in I}Ra_{i}+\mathfrak{m}^{2}\subseteq\mathfrak{m}$. Thus $\mathfrak{m}=\sum\limits_{i\in I}Ra_{i}+\mathfrak{m}^{2}$. 
This yields that $\mathfrak{m}(\mathfrak{m}/J)=\mathfrak{m}/J$ where $J:=\sum\limits_{i\in I}Ra_{i}$. Then by the graded version of the Nakayama lemma (which asserts that if $R$ is an $\mathbb{N}$-graded ring with the irrelevant ideal $R_{+}$ and $M$ is an $\mathbb{N}$-graded $R$-module with $R_{+}M=M$, then $M=0$), we have $\mathfrak{m}/J=0$ and so 
$\mathfrak{m}=\sum\limits_{i\in I}Ra_{i}$. \\
(i)$\Rightarrow$(iv): By assumption, the canonical map $K[x_{i}: i \in I]\rightarrow R$ is surjective.
By setting $\deg(x_{i}):=\deg(a_{i})$ for all $i$, we make the polynomial ring $K[x_{i}: i \in I]$ into an $\mathbb{N}$-graded ring whose base subring is $K$.
With this, the canonical map is a morphism of graded rings, so its kernel is a graded ideal.
Let $M$ be the graded ideal $(x_{i}: i\in I)$ of $K[x_{i}: i \in I]$.
If the kernel is not contained in $M^{2}$ then we can choose a homogeneous element of the form $f=rx_{d}-g$ in the kernel such that $0\neq r\in K$ and $g=g(x_{i_{1}},\ldots,x_{i_{n}})$ has no term involving $x_{d}$ (by our choice, $g$ has no term of the form $r'x_{d}$ with $r'\in K$, and if $g$ had a term $x_ih$ for some $h \not \in K$ then $f$ would not be homogeneous).
Then in $R$ we will have $ra_{d}-g(a_{i_{1}},\ldots,a_{i_{n}})=0$.
Since $K$ is a field, we may write $a_{d}=r^{-1}g(a_{i_{1}},\ldots,a_{i_{n}})$ which shows that $S\setminus\{a_{d}\}$ generates $R$ as a $K$-algebra, contradicting minimality of $S$. \\
(iv)$\Rightarrow$(i): Surjectivity yields that $R=R_{0}[S]$.
To prove minimality, suppose there is a proper subset $S'$ of $S$ such that $R=R_{0}[S']$.
Then choose some $a_{d}\in S\setminus S'$.
Thus there exists some polynomial $f(x_{i_{1}},\ldots,x_{i_{n}})\in R_{0}[x_{i}: i\in S']$ such that $a_{d}=f(a_{i_{1}},\ldots,a_{i_{n}})$.
This shows that the polynomial $x_{d}-f(x_{i_{1}},\ldots,x_{i_{n}})$ is in the kernel of the above canonical map which is contained in the square of the ideal $M=(x_{i}: i\in S)$.
It follows that $g:=x_{d}-c_{1}x_{i_{1}}+\ldots+c_{n}x_{i_{n}}\in M^{2}$ with $c_{j}\in R_{0}=K$.
Note that $d\notin\{i_{1},\ldots,i_{n}\}$.
Thus all of the coefficients of $g$ are zero.
In particular, $1=0$ which is a contradiction.
\\
Finally, every graded ring is generated (even additively) by all its homogeneous elements, and every generating set of a vector space has a subset as a basis.
Hence, there always exists such a set $S$ satisfying the above equivalent conditions.
\end{proof}

\begin{theorem} \label{thm:graded} Let $R=\bigoplus\limits_{n\geqslant0}R_{n}$ be an $\mathbb{N}$-graded ring with $R_{0}=K$ a finite field.
If $R$ is an avoidance ring, then it is a PIR.
\end{theorem}

\begin{proof} By Lemma \ref{lemma T-C}, we may write $R \cong K[X]/I$, where $X$ is a set of indeterminates over $K$ and $I$ is an ideal in $K[X]$ contained in $\m^2$, where $\m = (X)$ is the maximal ideal of $K[X]$.
If $R$ is not a PIR, then necessarily $|X| \geqslant 2$ (any quotient of a univariate polynomial ring over $K$ is a PIR).
Since there exists a surjection $R \twoheadrightarrow K[X]/\m^2$, by Example \ref{ex:nonavoidance}, $R$ does not have avoidance. \\
In fact, one sees from this that the only $\N$-graded $K$-algebras which have avoidance (for $K$ a finite field) are $K[x]$ and $K[x]/(x^i)$ for some $i \geqslant 1$.
\end{proof}

For an $\mathbb{N}$-graded ring $R=\bigoplus\limits_{n\geqslant0}R_{n}$ with $R_{0}=K$ a field, using \cite[Theorem 2.3]{Quartararo-Butts}, then $R$ is a module-avoidance ring if and only if the field $K$ is infinite.
In particular, if $\mathfrak{m}$ is a maximal ideal of a ring $R$ then the associated graded ring  $\gr_{\mathfrak{m}}R=
\bigoplus\limits_{d\geqslant0}\mathfrak{m}^{d}/\mathfrak{m}^{d+1}$ is a module-avoidance ring if and only if the field $R/\mathfrak{m}$ is infinite. 

\section{The avoidance property along ring maps}

We next turn towards avoidance in the relative setting, i.e. its behavior along a ring map.
Recall that a ring map $\phi : R \to S$ is an epimorphism if and only if for each $S$-module $M$ the canonical morphism of $S$-modules $M \otimes _{R}S \to M$ given by $x \otimes s \mapsto sx$ is injective (in fact, an isomorphism). In the literature (see e.g. \cite[Tag 04VM]{de Jong}, \cite{Samuel}, \cite{Tarizadeh 4}, \cite{Tarizadeh 5} and \cite{Tarizadeh}), one can find many properties and various equivalences of epimorphisms of rings. Note that this notion encompasses both surjective ring maps and localization maps. 

\begin{theorem}\label{Theorem 2} Let $\phi : R \to S$ be a flat ring epimorphism.
If $R$ is an avoidance ring, then $S$ is as well.
\end{theorem}

\begin{proof} It is well known that every ideal of $S$ is extended under $\phi$ (cf. \cite{Samuel} or \cite[Theorem 2.5(ii)]{Tarizadeh}).
Here we provide a short proof for completeness.
Let $J$ be an ideal of $S$, and set $I := \phi^{-1}(J)$.
Clearly $IS \subseteq J$.
For the reverse inclusion, note that the canonical ring map $\phi' : R/I \to S/J$ given by $r+I \mapsto \phi(r)+J$ is injective.
This gives an injective map $\phi' \otimes 1_S : R/I \otimes _R S \to S/J \otimes _R S$, because $S$ is $R$-flat.
Thus the canonical ring map $f : S/IS \to S/J$ given by $s+IS \mapsto s+J$ factors as $S/IS \xrightarrow{\cong} R/I \otimes _R S \xrightarrow{\phi' \otimes 1_S} S/J \otimes _R S \xrightarrow{\cong} S/J$.
Hence $f$ is injective, so $J \subseteq IS$.
Now the assertion follows from Lemma \ref{Lemma 1}.
\end{proof}

The above result, in particular, yields an alternative proof to the fact that the avoidance property is preserved by localizations (see Corollary \ref{Corollary 1}). Indeed, for any multiplicative subset $S$ of a ring $R$, the canonical ring map $R\rightarrow S^{-1}R$ is a typical example of a flat epimorphism. \\

We now seek dual versions of Lemma \ref{Lemma 1} and Theorem \ref{Theorem 2}.
If $\phi : R \to S$ is a ring map and $I, I_1,\ldots,I_n$ are ideals of $R$ with $I \subseteq \bigcup \limits_{k=1}^n I_k$, then $\phi(I) \subseteq \phi(\bigcup\limits_{k=1}^{n}I_k)=
\bigcup\limits_{k=1}^{n}\phi(I_k)$.
However, the extension ideal $I^{e} = IS$ is not necessarily contained in $\bigcup\limits_{k=1}^{n}(I_k)^{e}$.
This leads us to the following definition.

\begin{definition} We say that a ring map $\phi:R\to S$ has \emph{avoidance} if whenever $I, I_1,...,I_n$ are finitely many ideals of $R$ with $I\subseteq\bigcup\limits_{k=1}^{n} I_k$, then $IS\subseteq I_kS$ for some $k$. We also say that a ring $S$ has \emph{strong avoidance} if every ring map $\phi : R \to S$ has avoidance.
\end{definition}

Strong avoidance implies avoidance by considering the identity ring map.
However, the converse does not hold: there are rings which are avoidance rings, but do not have strong avoidance (see Example \ref{Example 1}).
Similarly, if $R$ is an avoidance ring, then every morphism of rings $R\to S$ has avoidance. The converse also holds by considering the identity map.
But its dual does not hold: if $S$ is an avoidance ring then a given ring map $R \to S$ does not necessarily have avoidance (see Example \ref{Example 1}). If a ring map $R\to S$ has avoidance then for any ring map $S\to T$ the composite $R\to T$ has avoidance.

\begin{proposition}\label{Proposition 1} Every absolutely flat ring has strong avoidance.
\end{proposition}

\begin{proof} Consider a ring map $\phi : R \to S$ and finitely many ideals $I, I_1,\ldots,I_n$ of $R$ with $S$ an absolutely flat ring.
If $I \subseteq \bigcup \limits_{k=1}^n I_k$, then by \cite[Theorem 2.2]{Tarizadeh-Chen} or by Theorem \ref{Theorem 7 McCoy}(iii), $I \subseteq \sqrt{I_k}$ for some $k$.
This yields that $I^{e} \subseteq (\sqrt{I_k})^{e} \subseteq \sqrt{(I_k)^{e}}=(I_k)^{e}$, the latter equality follows from the fact every ideal of an absolutely flat ring is a radical ideal.
\end{proof}

\begin{remark}
The converse of the above result does not hold.
For example, we show that the ring of integers $\mathbb{Z}$ has strong avoidance.
Let $\phi : R \to\mathbb{Z}$ be a ring map.
Since $\mathbb{Z}$ is the initial object in the category of rings, there exists a unique ring map $\psi : \mathbb{Z} \to R$ such that $\phi\psi$ is the identity map.
Thus $\phi$ is surjective, so for any ideal $I$ of $R$, one has $I^e = \phi(I)$.
Now if $I, I_1,\ldots,I_n$ are ideals of $R$ with $I \subseteq \bigcup \limits_{k=1}^n I_k$, then $I^{e}=\phi(I) \subseteq \phi(\bigcup\limits_{k=1}^{n}I_k)=
\bigcup\limits_{k=1}^{n}\phi(I_k)=\bigcup\limits_{k=1}^{n}(I_k)^{e}$.
But $I^{e}$ is a principal ideal since $\mathbb{Z}$ is a PID, so $I^{e} \subseteq (I_k)^{e}$ for some $k$.
\end{remark}

In order to correctly dualize Theorem \ref{Theorem 2}, one must find a suitable class of monomorphisms.
Recall that a ring map $\phi : R \to S$ is called a \emph{pure morphism} (or \emph{universally injective}) if for every $R$-module $M$, the induced map $1_{M} \otimes \phi : M \otimes _{R}R \to M \otimes _{R}S$, or equivalently, the map
$M \to M \otimes _{R}S$ given by $x \mapsto x \otimes 1$ is injective.
Pure morphisms have very nice properties, and have been extensively studied in the literature since Grothendieck (see e.g. \cite[Tag 08WE]{de Jong}). \\

By restricting to cyclic modules $M = R/I$, we obtain a notion of \emph{cyclically pure morphism} (following terminology of Hochster \cite{Hochster}): we say that a ring map $\phi : R \to S$ is \emph{cyclically pure} if for each ideal $I$ of $R$ the induced ring map $R/I \to S/IS$ given by $r+I \mapsto \phi(r)+IS$ is injective.
In particular, a cyclically pure morphism is injective.
Note that a ring map is cyclically pure if and only if every ideal of the source ring is contracted under this map.


\begin{proposition}
Let $\phi : R \to S$ be a cyclically pure ring map.
Then: \\
$\mathbf{(i)}$ If $S$ is absolutely flat, then $R$ is as well. \\
$\mathbf{(ii)}$ If $S$ is a valuation ring, then $R$ is as well. \\
In particular, in both cases above $R$ is an avoidance ring.
\end{proposition}

\begin{proof}
$\mathbf{(i)}$: If $S$ is absolutely flat, then every ideal of $S$ is radical.
Since the contraction of a radical ideal is radical, this shows that every ideal of $R$ is radical, so $R$ is absolutely flat. \\
$\mathbf{(ii)}$: If $S$ is a valuation ring, then $S$ is a domain whose ideals are totally ordered under inclusion.
Since $\varphi$ is injective, $R$ is also a domain, and any two ideals $I_1, I_2$ of $R$ are contractions of comparable ideals $J_1, J_2$ of $S$, hence $I_1, I_2$ are comparable in $R$, so $R$ is a valuation ring.
\end{proof}

Finally, we are ready to state the dual version of Lemma \ref{Lemma 1}:

\begin{lemma} \label{Theorem 3} Let $\phi : R \to S$ be a cyclically pure ring map.
If $S$ has strong avoidance, then $R$ has strong avoidance.
\end{lemma}

\begin{proof} Let $h : R' \to R$ be any ring map and $I, I_1,\ldots,I_n$ ideals of $R'$ with $I \subseteq \bigcup \limits_{k=1}^n I_k$.
Considering extensions of ideals under the ring map $\phi h : R' \to S$, we have $IS \subseteq I_{k}S$ for some $k$, because $S$ has strong avoidance.
Since $\phi$ is cyclically pure, $IR = \phi^{-1}\big((IR)S\big) = \phi^{-1}(IS)$, and similarly $I_kR = \phi^{-1}(I_kS)$.
Therefore $IR \subseteq I_{k}R$, and hence $R$ has strong avoidance.
\end{proof}

\begin{corollary} \label{Corollary 3} Let $\phi : R \to S$ be a pure morphism.
If $S$ has strong avoidance, then $R$ has strong avoidance.
\end{corollary}

\begin{proof}
Every pure morphism is cyclically pure, so the assertion follows from Lemma \ref{Theorem 3}.
\end{proof}

\begin{remark}
As special cases of Corollary \ref{Corollary 3}, we have the following examples of pure morphisms: \\
$\mathbf{(i)}$ Every faithfully flat ring map is pure. \\
$\mathbf{(ii)}$ If a ring map $\phi:R\rightarrow S$ splits as an $R$-module (i.e. there exists a morphism of $R$-modules $\psi : S \to R$ with $\psi \phi = 1_R$), then it is pure.
For example, for any ring map $\phi:R\rightarrow S$, the canonical ring map $S \to S \otimes_R S$ given by $s \mapsto s \otimes 1$ splits as an $S$-module.
As another example, by the direct summand theorem (see \cite{Andre} or \cite[Theorem 5.4]{Bhatt}), any finite ring map $R \to S$ with $R$ a regular Noetherian ring splits as an $R$-module.
\end{remark}

\begin{example}\label{Example 1} The assumption of ``strong avoidance'' in Lemma \ref{Theorem 3} is crucial.
In other words, the naive dual statement of Lemma \ref{Lemma 1} does not hold.
For example, let $K$ be a finite field, $R$ a $K$-algebra which is not an avoidance ring (see e.g. Example \ref{ex:nonavoidance}) and $F$ an infinite field containing $K$ as a subfield (e.g. $F = \overline{K}$, the algebraic closure of $K$).
Then the canonical ring map $R \to R \otimes _{K}F$ given by $r \mapsto r \otimes 1_{F}$ is faithfully flat, since it makes $R \otimes _{K}F$ a nonzero free $R$-module.
Moreover, note that the ring $R \otimes_K F$ is an avoidance ring by Corollary \ref{cor:exsAvoidance}(iii).
Also note that this ring map $R \to R \otimes _{K}F$ does not have avoidance, because if it has avoidance then $R$ will be an avoidance ring which is impossible.
In particular, for a finite field $K$ the canonical ring map $K[x,y]\rightarrow\overline{K}[x,y]$ does not have avoidance.
Finally, note that ``strong avoidance'' and ``module-avoidance ring'' are quite distinct concepts.
For example, $R \otimes_K F$ is a module-avoidance ring, because it contains the infinite field $F$ as a subring, but it does not have strong avoidance by Corollary \ref{Corollary 3}.
Conversely, we observed that the ring of integers $\mathbb{Z}$ has strong avoidance, but is not a module-avoidance ring.
\end{example}

We conclude this article with the following problem. 

\begin{conjecture} A ring $R$ has strong avoidance if and only if for each prime ideal $\mathfrak{p}$ of $R$, the ring $R_{\mathfrak{p}}$ is a PID.
\end{conjecture}

\textbf{Acknowledgments.} We would like to give sincere thanks to Professor Christian Gottlieb, who generously shared with us his valuable ideas. We also thank Professor Neil Epstein, who pointed out the notion of cyclic purity, which led us to an invaluable literature search which greatly changed the content from a previous version. Last but not least, we give our deep gratitude to the referee for a very careful reading of the article.

\end{document}